\documentclass{amsart}
\usepackage{amsfonts}

\usepackage{graphicx}
\usepackage{amscd}
\usepackage{amsmath}
\usepackage{amsxtra}

\newcommand{\strutstretchdef}{\newcommand{\strutstretch}{\vphantom{\raisebox{1pt}{$\big($}\raisebox{-1pt}{$\big($}}}}
\theoremstyle{plain}
\newtheorem{theorem}{Theorem}[section]

\newtheorem{proposition}[theorem]{Proposition}
\newtheorem{corollary}[theorem]{Corollary}

\theoremstyle{definition}
\newtheorem{definition}[theorem]{Definition}
\newtheorem{example}[theorem]{Example}

\theoremstyle{remark}
\newtheorem{remark}[theorem]{Remark}

\numberwithin{equation}{section}

\newlength{\struh}
\newlength{\textminustop}
\strutstretchdef
\DeclareMathOperator{\shift}{shift}

\DeclareMathOperator{\Exp}{Exp}

\begin{document}
\title[Moment Infinitely Divisible Weighted Shifts]{Moment Infinitely Divisible Weighted Shifts}
\author{Chafiq Benhida}
\address{UFR de Math\'{e}matiques, Universit\'{e} des Sciences et Technologies de Lille, F-59655 Villeneuve d'Ascq Cedex, France}
\email{Chafiq.Benhida@math.univ-lille1.fr}
\author{Ra\'{u}l E. Curto}
\address{Department of Mathematics, The University of Iowa, Iowa City, Iowa 52242, USA}
\email{raul-curto@uiowa.edu}
\author{George R. Exner}
\address{Department of Mathematics, Bucknell University, Lewisburg, Pennsylvania 17837, USA}
\email{exner@bucknell.edu}
\thanks{The first named author was partially supported by Labex CEMPI (ANR-11-LABX-0007-01); the second named author was partially supported by NSF Grant DMS-1302666. \ The third named author was partially supported by Labex CEMPI (ANR-11-LABX-0007-01).}

\begin{abstract}
We say that a weighted shift $W_\alpha$ with (positive) weight sequence $\alpha: \alpha_0, \alpha_1, \ldots$ is {\it moment infinitely divisible} (MID) if, for every $t > 0$, the shift with weight sequence $\alpha^t: \alpha_0^t, \alpha_1^t, \ldots$ is subnormal. \ Assume that $W_{\alpha}$ is a contraction, i.e., $0 < \alpha_i \le 1$ for all $i \ge 0$. \ We show that such a shift $W_\alpha$ is MID if and only if the sequence $\alpha$ is log completely alternating. \  This enables the recapture or improvement of some previous results proved rather differently. \ We derive in particular new conditions sufficient for subnormality of a weighted shift, and each example contains implicitly an example or family of infinitely divisible Hankel matrices, many of which appear to be new.
\end{abstract}

\maketitle


\section{Introduction}

Let $\mathcal{H}$ be a separable complex Hilbert space and $\mathcal{L}(\mathcal{H})$ the algebra of bounded linear operators on $\mathcal{H}$. \  An operator $T$ in $\mathcal{L}(\mathcal{H})$ is {\it normal} if it commutes with its adjoint ($T^* T = T T^*$) and {\it subnormal} if it is (unitarily equivalent to) the restriction of a normal operator to an invariant subspace (closed linear manifold). \  One standard condition for subnormality is the Bram-Halmos condition (\cite{Br}) requiring non-negativity of operator matrices:  $T$ is subnormal if and only if, for each $k = 1, 2, \ldots$, the $(k+1) \times (k+1)$ operator matrix $({T^*}^{i}T^j)_{0\leq i,j \leq k}$  is non-negative. \  (That is, $T$ is $k${\it -hyponormal} for each $k=1, 2, \ldots$.)  Alternatively, the Agler-Embry condition (\cite{Ag}) is that, for an operator $T$ which is contractive ($\|T\| \leq 1$), $T$ is subnormal if and only if, for each $n = 1, 2, \ldots$, one has $\sum_{i=0}^n (-1)^i \binom{n}{i} {T^*}^i T^i \geq 0$. \  (That is, $T$ is $n${\it -contractive} for each $n=1, 2, \ldots$.)

We consider here weighted shifts, so let  $\ell^2$ be given its standard basis $\{e_j\}_{j=0}^\infty$. \  Given a
 weight sequence $\alpha: \alpha_0,\alpha_1, \alpha_2,
\ldots,$
 we define the weighted shift $W_\alpha$ on $\ell^2$ by $W_\alpha e_j := \alpha_j e_{j+1}$ and extend by linearity. \ The moments are $\gamma_0 := 1$ and $\gamma_j :=
\prod_{i=0}^{j-1} \alpha_i^2$, $j \geq 1.$ \ For virtually all questions of interest, it is without loss of generality to assume that the $\alpha_j$ are positive (see \cite{Sh}), and we do henceforth.

The subnormality of unilateral weighted shifts can be characterized in terms of their moments, as follows: $W_{\alpha}$ is subnormal if and only if there exists a Borel probability measure $\mu$ supported on the interval $[0, \left\|W_\alpha\right\|^2]$ such that $\gamma_n = \int t^n \; d\mu(t)$ for every $n=0,1,\cdots$; this is due to C. Berger \cite[III.8.16]{Con} and was independently established by R. Gellar and L.J. Wallen \cite{GeWa}; the measure $\mu$ is said to be the Berger measure of $W_{\alpha}$. \ On the other hand, G. Exner proved that $W_{\alpha}$ is $n$-contractive if and only if $\sum_{j=0}^{n}(-1)^{j}\binom{n}{j} \gamma_{m+j} \geq 0, \hspace{.2in}m =0, 1, \ldots $ \cite{Ex}. \ As yet another moment characterization, the Bram-Halmos conditions yield, for a weighted shift, the necessary and sufficient condition that the Hankel matrices $H(n,k):=(\gamma_{n+i+j})_{i,j=0}^k$ be positive semi-definite for all $n \ge 0, \; k \ge 1$ \cite{Cu1}.

The Aluthge transform has been an object of recent study;  recall that if $T$ is any operator with $T = U|T|$ its polar decomposition,
the Aluthge transform is defined by $AT(T) := |T|^{\frac{1}{2}} U |T|^{\frac{1}{2}}$. \  Define the iterated Aluthge transform $AT^n(\cdot)$ by
$AT^{n+1}(T) = AT(AT^n(T))$. \  It is an easy computation that the Aluthge transform of a weighted shift $W_\alpha$ is a weighted shift with weight sequence $\sqrt{\alpha_0 \alpha_1}, \sqrt{\alpha_1 \alpha_2}, \ldots$. \

Closely related to the subnormality of $W_{\alpha}$ are the (possible) subnormalities of the square root shift $W_{\sqrt{\alpha}}:=\shift\;(\sqrt{\alpha _{0}},\sqrt{\alpha _{1}},\sqrt{\alpha _{2}},\cdots )$ and the Aluthge transform of $W_{\alpha}$. \ In fact, $AT(W_{\alpha})$ is the Schur product of $W_{\sqrt{\alpha}}$ and its restriction to the subspace of $\ell^2$ spanned by the basic vectors $e_1,e_2,\cdots$. \ It follows that a sufficient condition for the subnormality of both $W_{\alpha}$ and $AT(W_{\alpha})$ is the subnormality of $W_{\sqrt{\alpha}}$.

In \cite{CE}, the authors studied the connection between the subnormality of $W_{\alpha}$ and $W_{\sqrt{\alpha}}$, and in particular how the associated Berger measures $\mu$ and $\nu$, when they exist, are linked. \ This leads to the so-called Square Root Problem: \

Given a probability measure $\mu$ (supported on a compact interval in $\mathbb{R} _+$), does there exist $\nu$ satisfying
\begin{equation}    \label{eq:momentmatching}
\int t^n \, d\mu(t) = \left(\int t^n \, d\nu(t) \right)^2, \hspace{.2in} n = 0, 1, 2, \ldots  \hspace{.3in} .
\end{equation}
If $\nu$ exists, can one find it? \ A related question is: Given $0<p<1$ consider the moment matching equations
\begin{equation}
\int t^n \, d\mu(t) = \left(\int t^n \, d\nu(t) \right)^p, \hspace{.2in} n = 0, 1, 2, \ldots  \hspace{.3in} .
\end{equation}
One then tries to find one measure, given the other. \ Finally, using the setting of convolution of measures, one might ask: Given a measure $\mu$ as above and a nonnegative integer $N$, can one find a measure $\nu$ such that $\mu=(\nu)^N:=\nu \ast \nu \ast \cdots \ast \nu \; (N \textrm{times})$?

\begin{definition} \label{MIDdef}
A unilateral weighted shift $W_\alpha$ with (positive) weight sequence $\alpha: \alpha_0, \alpha_1, \ldots$ is {\it moment infinitely divisible} (MID) if, for every $t > 0$, the shift with weight sequence $\alpha^t: \alpha_0^t, \alpha_1^t, \ldots$ is subnormal.
\end{definition}

In this paper we prove that if $W_\alpha$ is a contraction then it is MID if and only if the sequence $\alpha$ is log completely alternating.

(There is a related notion of ``infinite divisibility," concerning divisibility with respect to convolution, arising in probability theory in connection with L\'evy processes; see, for example, \cite{Sa}. \ We use the language ``moment" infinite divisibility (MID) to avoid confusion with this latter concept.)

It is well known that the Schur (entry-wise) product of positive matrices is again positive, and from this and the Bram-Halmos condition (in its Hankel moment matrix formulation for weighted shifts) it is immediate that if $W_\alpha$ is a subnormal shift, then for each positive integer $m$, the shift with weight sequence $(\alpha_i^m)$ is subnormal. \  (See \cite{CuP} for another application of this approach.) \ Consistent with Definition \ref{MIDdef}, a positive matrix $(a_{i j})$ is called \textit{infinitely divisible} if, for each $p>0$, the matrix $(a_{i j}^p)$ (the ``Schur $p$-th power'') is again positive (see, for example, \cite{Bh}). \ To a certain extent, it is this fact, taken with the Bram-Halmos condition, which motivates our terminology of moment infinitely divisible weighted shift. \ In what follows, observe that when we produce an example of an infinitely divisible weighted shift or family of such shifts, we are implicitly producing an infinitely divisible Hankel matrix (of its moments) or family of such matrices, and the matrices produced seem not to be in the literature of infinitely divisible matrices.  We do not comment separately on this below.

It is an easy computation to realize that raising weights to the $t$-th power is equivalent to raising moments to the $t$-th power, and we will say as well that a weight or moment sequence is infinitely divisible with the obvious meaning. \ It is known from \cite{Ex2}, using an approach based upon completely monotone functions, that certain weighted shifts called the Agler shifts to be defined below (one of which is the familiar Bergman shift) have this property.

To describe our main result, we need to recall the definition of completely alternating sequence, which in turn requires the definition of completely monotone sequence, and we begin in fact with the definition of a completely monotone function.

\begin{definition}
A function $f: \mathbb{R}_+ \rightarrow \mathbb{R}_+ \setminus \{0\}$ is {\em completely monotone} if its derivatives alternate in sign: $f^{(2 j)} \ge 0$ for $j \geq 1$, and $f^{(2 j + 1)} \le 0$ for $j \geq 0$. \
\end{definition}

The following statements are well known:

\begin{itemize}
\item \ The function $f$ is completely monotone if and only if $f = \mathcal{L}(\mu)$ for some positive measure $\mu$, where $\mathcal{L}$ denotes the Laplace transform. \newline
\item \ If a completely monotone function $f$ interpolates the moments of a unilateral weighted shift, i.e., $f(n) = \gamma_n$ for all $n \ge 0$, then the shift is subnormal.
\end{itemize}

We pause to recall that some well known weighted shifts are MID. \ The Agler shifts $A_j$, $j = 1, 2, \ldots$, are those with weight sequence $\sqrt{\frac{n+1}{n+j}}$, $n = 0, 1, \ldots$. \  (These were used in \cite{Ag} as part of a model theory for hypercontractions.) \ Since $A_1$ is the unilateral shift, its Berger measure is $\delta_{1}$; for $j \ge 2$ the Berger measure of $A_j$ is $d \mu(t)=(j-1)(1-t)^{j-2}dt$ on $[0,1]$.

\begin{theorem} (\cite{Ex2}) \ For $j = 2, 3, \ldots$, let $A_j$ be the $j$-th Agler shift. \ Any $p$-th power transformation ($p > 0$) of $A_j$ is subnormal, as is any $m$-th iterated Aluthge transform of $A_j$.
\end{theorem}

The proof (which uses monotone function theory) offers no information about the Berger measure of the resulting shift; however, it brings to the fore the significance of complete monotonicity in the study of MID shifts.

To proceed, we need to introduce the {\it difference operator} $\nabla$, acting on sequences $\varphi$:
$$\nabla^0 \varphi := \varphi, \; \; \;  \; \; \; \; (\nabla \varphi)(n) := \varphi(n) - \varphi(n+1),$$
$$\nabla^k \varphi := \nabla \nabla^{k-1} \varphi,$$
for all $k \geq 1$. \

For example,

$$(\nabla \varphi)(n)= \varphi(n)-\varphi(n+1),$$
$$(\nabla^2 \varphi)(n)= \varphi(n)-2 \varphi(n+1)+\varphi(n+2),$$
$$(\nabla^3 \varphi)(n)= \varphi(n)-3 \varphi(n+1)+3 \varphi(n+2)-\varphi(n+3),$$
and so on. \

\begin{definition}
A sequence $\varphi$ is said to be {\it completely monotone} if $(\nabla^k \varphi)(n) \geq 0$ for all $k, n \geq 0$. \
\end{definition}

\begin{remark}
Looking at the Agler-Embry conditions for subnormality, we easily see that a contractive shift $W_{\alpha}$ is subnormal if and only if $(\gamma_n)$ is completely monotone if and only if $(\nabla^k \gamma)(n) \ge 0 \; \; (\textrm{all} \; k,n \ge 0)$.
\end{remark}

\begin{definition}   \label{def:CA}
A sequence $\psi$ is said to be {\it completely alternating} if $(\nabla^k \psi)(n) \leq 0 \; \; (\textrm{for all} \; k \ge 1, n \ge 0)$; equivalently, if the sequence $-\nabla \psi$ is completely monotone. \ Similarly, given an integer $k \ge 1$, a sequence $\psi$ is said to be {\it $k$-alternating} if $(\nabla^k \psi)(n) \leq 0$ for all $n \geq 0$, and is $k$-hyperalternating if it is $m$-alternating for all $m \le k$.
\end{definition}

The set of completely alternating sequences will be denoted by $\mathcal{CA}$, and the subset of sequences in $\mathcal{CA}$ with all positive terms will be denoted by $\mathcal{CA}_+$. \ The set of positive-term sequences whose logarithms are in $\mathcal{CA}$ will be denoted by $\Exp \mathcal{CA}$, and the sequences called ``log completely alternating.''

There is a well known link between completely alternating and completely monotone sequences, which we now state.

\begin{proposition}
\ (\cite[Prop. 6.10]{BCR}) \ The sequence $\psi$ is completely alternating if and only if the sequence $\varphi_t:=e^{-t \psi}$ is completely monotone for all $t>0$.
\end{proposition}

\begin{proposition} \ (\cite[Chapter 4, Prop. 6.12]{BCR}) \ The sequence $\psi$ is completely alternating if and only if it has an associated L\'evy-Khintchin measure; that is,
\begin{equation} \label{Levy}
\psi(n)=a + bn+ \int_0^1 (1-t^{n}) \; d\mu(t),
\end{equation}
where $\mu \ge 0$. \ (Following the usual convention, we declare $t^0=1$ for all $t$ and $\int_0^1 0 \; d\mu(t)=0$ for any $\mu$.)
\end{proposition}

\begin{remark} \label{remAgler}
In the specific case of the Agler shifts $A_j$, the sequence of weights squared is
$$
\frac{1}{j} + \int_0^1 (1-t^{n}) (j-1)t^{j-1} \; dt.
$$
Thus, while the Lebesgue measure $(j-1)t^{j-2} dt$ helps recover the moments of the Agler shift $A_j$, the L\'evy-Khintchin identity (\ref{Levy}) (with $b=0$) reproduces the actual weights squared. \ Note that in this (atypical) case, the L\'evy-Khintchin measure is obtained from the Berger measure by multiplying by $t$ and taking $a$ to be the zero-th weight.
\end{remark}

Completely alternating sequences have been studied extensively, but it is important to note that usually there is the assumption that the sequence is positive (the study is in the context of semigroups, in particular $\mathbb{R}_+$) while here we allow negative terms. \  Part of what separates our results from some results of Athavale and others in \cite{At}, \cite{SA}, and \cite{AtRa} is this modest change. \  It is easy to check that both $\mathcal{CA}$ and $\mathcal{CA}_+$ are closed under sums, the addition of constants, and (positive) scalar multiples, while $\mbox{\rm Exp}\,  \mathcal{CA}$ is closed under (positive) scalar multiplication, products, and positive powers.

\section{Some Auxiliary Results}

For later use, we record a few auxiliary results.

\begin{proposition}  \label{prop:log1pphi} \ (\cite[Ch. 3, Cor. 2.10]{BCR}; see also \cite[Cor. 1]{AtRa}) \
Let $\varphi \in \mathcal{CA}_+$. \ Then $\ln(1 + \varphi) \in \mathcal{CA}_+$.
\end{proposition}

\begin{proposition} \label{propBCR} \ (\cite[Cor. 1]{AtRa}; cf. \cite[Ch. 3, Cor. 2.10]{BCR}) \ Let $\varphi \in \mathcal{CA}_+$. \ Then $\varphi^p \in \mathcal{CA}_+$ for all $0 < p < 1$.
\end{proposition}

\begin{remark}
Since $(\frac{n+1}{n+2}) \in \mathcal{CA}_+$, it follows from Proposition \ref{propBCR} that $(\sqrt{\frac{n+1}{n+2}}) \in \mathcal{CA}_+$.
\end{remark}

We now prove that each positive-term completely alternating sequence is automatically log completely alternating.

\begin{proposition}  \label{CApvsLogCA}
The containment $\mathcal{CA}_+ \subseteq \mbox{\rm Exp}\,  \mathcal{CA}$ holds.
\end{proposition}

\begin{proof}
Consider first some sequence $(\beta_n)$ in $\mathcal{CA}_+$ all of whose terms are strictly greater than $1$. \  Then with $\alpha_n := \beta_n - 1$ we get from Proposition \ref{prop:log1pphi}  that $(\ln(1 + \alpha_n)) = (\ln(\beta_n))$ is in $\mathcal{CA}_+$, which is to say that $(\beta_n)$ is log completely alternating. \  Now consider some sequence $(\delta_n)$ in $\mathcal{CA}_+$;  since it is completely alternating, it is nondecreasing and is in particular bounded away from zero. \  Then there exists some $k>0$ such that $(k \delta_n)$ has all terms strictly greater than one (and of course $(k \delta_n)$ is still completely alternating). \  But then $(k \delta_n)$ is log completely alternating by what was just shown, and therefore $(\delta_n)$ is log completely alternating as well.
\end{proof}

The reverse containment does not hold: recall the Agler shifts defined above and that their sequences of weights squared are completely alternating as in Remark 1.9.

\begin{proposition}
We have $\mbox{\rm Exp}\,  \mathcal{CA} \not\subseteq \mathcal{CA}$.
\end{proposition}

\begin{proof}
The sequence $\left((\frac{n+1}{n+2})^3\right)$ is log completely alternating but not completely alternating. \  As just noted, $A_2$ -- the Bergman shift -- with weight sequence $\left(\sqrt{\frac{n+1}{n+2}}\right)$, has weights squared completely alternating. \  Therefore the weights squared are log completely alternating, so clearly any power of them is log completely alternating. \  But one may check using \cite{Wol} that $\left((\frac{n+1}{n+2})^3\right)$ is not completely alternating.
\end{proof}

\begin{remark} \label{link}
We now present a link between moments and weights, using the difference operator $\nabla$. \ Recall that $\gamma_{n+1}=\gamma_{n} \alpha_n^2$ \; (all $n \ge 0$). \ It follows that
$$
\mbox{\rm ln}\,  \gamma_{n+1} - \mbox{\rm ln}\,  \gamma_{n} = \mbox{\rm ln}\,  \alpha_n^2
$$
or, equivalently,
$$
(\nabla \mbox{\rm ln}\,  \gamma)(n)=-(\mbox{\rm ln}\,  \alpha^2)(n)=-(\nabla^0 \mbox{\rm ln}\, \alpha^2)(n)
$$
for all $n \ge 0$. \ Using mathematical induction one can then prove that
\begin{equation} \label{eq21}
(\nabla^{k+1} \mbox{\rm ln}\,  \gamma)(n)=-(\nabla^k \mbox{\rm ln}\,  \alpha^2)(n) 
\end{equation}
for all $n \ge 0$. \ It follows that if $-\mbox{\rm ln}\, \gamma$ is completely alternating, then $\mbox{\rm ln}\, \alpha^2$ is completely alternating.
\end{remark}

\section{Main Results}

The following is the main result of this paper. \  It should be noted that heretofore conditions related to subnormality of shifts have been given on the sequence of moments (e.g., Bram-Halmos and Agler-Embry), rather than on the sequence of weights.

\begin{theorem} \label{main}
Let $W_\alpha$ be a contractive weighted shift with (positive) weight sequence $\alpha = (\alpha_n)$. \ Then  $W_\alpha$ is moment infinitely divisible (MID) if and only if $(\alpha_n^2)$ is log completely alternating (equivalently, $(\alpha_n)$ is log completely alternating).
\end{theorem}

\begin{proof}
First, we recall that a sequence $\psi$ is completely alternating if and only if \
\begin{equation} \label{eqexpo}
\varphi_t:=e^{-t \psi} \; \; \; \; \; \;
\end{equation}
is completely monotone for all $t>0$. \ Observe also that since $W_\alpha$ is a contraction, so is the shift obtained by raising the weights (equivalently, the moments) to the $t$-th power for any $t > 0$. Therefore, $W_{\alpha}$ is MID if and only if $\gamma^t$ is completely monotone for all $t>0$. \ Since $\gamma^t=e^{-t (-\mbox{\rm ln}\, \gamma)}$, we can let $\psi:=-\mbox{\rm ln}\, \gamma$ and $\varphi_t \equiv \gamma^t$ in (\ref{eqexpo}), and conclude that $W_{\alpha}$ is MID if and only if $-\mbox{\rm ln}\, \gamma$ is completely alternating.

($\Longrightarrow$) \ Assume that $W_{\alpha}$ is MID. \ Then $-\mbox{\rm ln}\, \gamma$ is completely alternating or, equivalently, $\nabla \mbox{\rm ln}\, \gamma$ is completely monotone (by Definition \ref{def:CA}). \ Then $(\nabla^{k+1} \mbox{\rm ln}\,  \gamma) \ge 0$ for all $k \ge 0$, and by (\ref{eq21}) we have
$$
- \nabla^k (\mbox{\rm ln}\,  \alpha^2) \ge 0
$$
for all $k \ge 0$. \ We thus have
$$
\nabla^{k} (\mbox{\rm ln}\,  \alpha^2) \le 0
$$
for all $k \ge 0$; that is,
$$
\mbox{\rm ln}\,  \alpha^2 \in \mathcal{CA}.
$$

($\Longleftarrow$) \ The proof of the converse is entirely similar.
\end{proof}

We may recapture two results obtained using interpolation of moments by completely monotone functions. \  Recall that the weighted shift denoted $S(a,b,c,d)$ (where $a, b, c, d >0$ and with $a d > b c$) has weights $\sqrt{\frac{a n + b}{c n + d}}$, and that in \cite{CPY} certain subshifts of such shifts are defined. \  Observe as well that if we throw away a finite number of terms at the beginning of a completely alternating (or log completely alternating) sequence what remains is still in the original class.  Therefore, some of the results to follow may be generalized easily to restrictions of shifts to the canonical invariant subspaces of finite co-dimension, which generalizations we leave to the reader.

\begin{corollary}[\cite{Ex2}]
The Agler shifts are all moment infinitely divisible, as are the contractive shifts $S(a,b,c,d)$  (with $a, b, c, d >0$ and $a d > b c$) and their subshifts.
\end{corollary}

\begin{proof}
As noted above, the sequence of weights squared for an Agler shift is completely alternating, hence log completely alternating. \  For the second claim, we may, by scaling, reduce to the case $a = c = 1$ to consider $S(1,s,1,t)$ with $t > s$. \  One computes (using induction and the recursion relating differences of length $n+1$ to differences of length $n$)
$$\nabla^m(\alpha^2_n) = \frac{m!(s-t)}{(n+t)(n+t + 1) \cdots (n+t+m)},$$
which is negative since $t > s$. \  The sequences of weights squared is then in $\mathcal{CA}_+$ and hence $\mbox{\rm Exp}\,  \mathcal{CA}$. \  It is easy to check that a subshift of some $S(a,b,c,d)$ is merely some $S(a',b',c',d')$ satisfying the needed conditions; the desired result then follows.
\end{proof}

We now have the following consequences, which seem difficult to obtain in other ways.  Recall that for a weighted shift which is subnormal (even hyponormal) the weight sequence is nondecreasing and $\|W\| = \sup_n \alpha_n$.

\begin{corollary}
Suppose $W_\alpha$ is a contractive weighted shift.
\begin{enumerate}
\item If $W_\alpha$ is MID, let $W_{\alpha'}$ be some perturbation consisting only of decreasing the zeroth weight of $W_\alpha$ (while leaving it positive). \  Then $W_{\alpha'}$ is MID;
\item If the weighted shift with weight sequence $\frac{e^\alpha}{e^{\|W\|}} = (\frac{e^{\alpha_0}}{e^{\|W\|}}, \frac{e^{\alpha_1}}{e^{\|W\|}}, \ldots)$ is MID,  then $W_\alpha$ is MID.
\end{enumerate}
\end{corollary}

\begin{proof}
For the first claim, a reduction in the first weight does not destroy log completely alternating. \  The second claim is immediate from Proposition \ref{CApvsLogCA} and Theorem \ref{main}.  Observe that for the conclusion in (ii) we could assume, instead, that the weight sequence $e^\alpha = (e^{\alpha_0}, e^{\alpha_1}, \ldots)$ is completely alternating.
\end{proof}

We may improve in some ways the result of \cite[Lemma, p. 217]{CoLo}, which proves subnormality for a certain weighted shift of importance in considering Toeplitz operators (existence of non-normal, non-analytic subnormal Toeplitz operators).

\begin{corollary}   \label{cor:Toeplitz}
If $0 < p < 1$, the weighted shift with weights $\alpha_n := (1-p^{2n+2})^{\frac{1}{2}}$ for $n = 0, 1, \ldots$ is subnormal and moment infinitely divisible. \ Further, if $p_1, p_2, \ldots, p_k$ is a collection of values in $(0,1)$, the weighted shift with weights $\alpha_n := \prod_{i=1}^k(1-p_i^{2n+2})^{\frac{1}{2}}$ for $n = 0, 1, \ldots$ is subnormal and MID.
\end{corollary}

\begin{proof}
For the first claim, it is clear that the weight sequence squared is generated as $\int_0^1 (1-t^{n+1}) d \delta_{p^2}(t)$ by the measure $\delta_{p^2}$ (where $\delta_x$ is the Dirac mass at $x$), and is therefore completely alternating. \  The second claim follows easily from the first using Schur products.
\end{proof}

Here is an example for which moment infinite divisibility and subnormality are easy from this point of view but seem difficult from approaches based upon the moments;  it is based on an example in \cite[Ex. 6.24, p. 139]{BCR}.

\begin{example}
The weighted shift with weights the sequence $(\alpha_n)$ defined by
$$\alpha_n = 1 + \frac{1}{2} + \frac{1}{3} + \ldots + \frac{1}{n+1} - \ln (n+2), \hspace{.2in}n=0,1, \ldots,$$
or the shift with weights the square roots of these, is MID and subnormal. \  This is because this sequence is completely alternating (and it increases to the Euler constant $\gamma \approx 0.577\ldots $).
\end{example}

Using the definition of moment infinitely divisible, and Schur products, one may show that if $W_\alpha$ is MID then so is $AT(W_\alpha)$, but we may also prove this as follows, in part because it provides a ``local'' version of a result to come later.

\begin{corollary}
If a contractive weighted shift $W_\alpha$ is moment infinitely divisible then so is $AT(W_\alpha)$.
\end{corollary}

\begin{proof}
Since $(\ln(\alpha_n))$ is completely alternating, clearly so is $(\ln(\sqrt{\alpha_n \alpha_{n+1}}))$.
\end{proof}

For $0 \le t \le \frac{1}{2}$, consider now the generalized mean transform of a linear operator $T$, defined as $\hat{T}(t) := 
\frac{1}{2}[\tilde T(t)+\tilde T(1-t) ]$, where $\tilde T(t):=|T|^tU|T|^{1-t}$ is the generalized Aluthge transform of $T$; when $t=0$, this is the mean transform, already considered in \cite{LLY}. \ For a weighted shift $W_\alpha$, $\widehat{W}_\alpha(t)$ turns out to be another weighted shift with weight sequence $\frac{\alpha_n^{1-t}\alpha_{n+1}^t + \alpha_n^t\alpha_{n+1}^{1-t}}{2} \; (n \ge 0)$. \  In \cite{LLY}, the authors conjecture that the mean transform of the Bergman shift is subnormal and prove that the mean transform of the Schur square of the Bergman shift (that is, the shift with weights the squares of those of the Bergman shift) is subnormal. \  Both of these results are obtained as corollaries of the following.

\begin{proposition}
Suppose $W_\alpha$ is a contractive weighted shift whose weights squared are completely alternating. \  Then both the weighted shift whose weights are $\alpha_n' := \sqrt{\frac{\alpha_n^2 + \alpha_{n+1}^2}{2}}$ and the generalized mean transform of $W_\alpha$ are (contractive and) MID and hence subnormal. \  If the weights of $W_\alpha$ are completely alternating then the generalized mean transform of $W_\alpha$ is (contractive and) MID and hence subnormal.
\end{proposition}

\begin{proof}
Suppose that $W_\alpha$ is a contractive weighted shift whose weights squared are completely alternating;  it is then clear that the two resulting shifts considered are contractions.  \  The squares of the $\alpha_n'$ clearly form a completely alternating (and hence log completely alternating) sequence. \  As well, since the weights squared of $W_\alpha$ are completely alternating, using the result for $p$-th roots for $p < 1$ so is the sequence of their square roots -- that is, the weight sequence for $W_\alpha$ -- and therefore so is the weight sequence for the generalized mean transform of $W_\alpha$. \  The second claim follows by the argument just used applied directly to the weight sequence for $W_\alpha$.
\end{proof}

\begin{corollary}
The mean transform of the Bergman shift and the mean transform of the Schur square of the Bergman shift are both moment infinitely divisible and hence subnormal.
\end{corollary}

It is clear that small modifications of the Aluthge and mean transforms (an average of three weights, for example) will yield similar results, and also that iterated Aluthge and mean transforms will yield the same conclusions. \  We consider next the Ces\`aro transform of a sequence:  recall that given a sequence $(x_n)$, its Ces\`aro transform is the sequence whose $n$-th term is
$$c_n = \frac{x_0 + x_1 + \ldots x_{n}}{n+1}.$$
We define as well the geometric Ces\`aro transform as the one whose $n$-th term is
$$g_n= \sqrt[n+1]{x_0 x_1 \cdots x_{n}}.$$
Given a sequence $(x_n)$, we will occasionally write $(C(x_n))$ for its Ces\`aro transform and $(GC(x_n))$ for its geometric Ces\`aro transform.
In some sense the mean and Aluthge transforms are, for weighted shifts, ``local'' versions of these.

We have the following.
\begin{proposition} \label{prop211}
If the sequence $(x_n)$ is completely alternating then its Ces\`aro transform $(C(x_n))$ is completely alternating. \  If the sequence $(x_n)$ is log completely alternating, its geometric Ces\`aro transform $(GC(x_n))$ is log completely alternating.
\end{proposition}

\begin{proof}
The second assertion follows easily from the first. \  For the first, let $D(m,j)$ denote the difference of order $m$ for the original sequence $(x_n)$ beginning at position $j$:
$$D(m,j) := \sum_{i=0}^m (-1)^i \binom{m}{i}x_{j+i}, \hspace{.2in}m = 1, 2, \ldots; j = 0, 1, \ldots .$$
Observe that since the original sequence is completely alternating, each of these is non-positive.
Similarly, let $C(m,j)$ denote the differences for the Ces\`aro transform:
$$C(m,j) := \sum_{i=0}^m (-1)^i \binom{m}{i}c_{j+i}, \hspace{.2in}m = 1, 2, \ldots; j = 0, 1, \ldots .$$
One may show by induction, and using the recursive relationships $D(m+1,j) = D(m,j)-D(m,j+1)$, their $C(m,j)$ equivalents, and some elementary relationships among binomial coefficients, that
$$C(m,j) = \frac{m!j!}{(m+j+1)!}\sum_{k=0}^j \binom{m+k}{m} D(m,k),$$
and therefore that the $C(m,j)$ are non-positive, as required.
\end{proof}

We leave to the interested reader the proof of the analogous stability results for the sequences defined by
$$
c_n^k:=\frac{x_n+x_{n+1}+ \cdots +x_{n+k}}{k+1}
$$
and
$$
g_n^k:=\sqrt[k+1]{x_n+x_{n+1}+ \cdots +x_{n+k}},
$$
for a fixed integer $k \ge 1$.

We have the following corollaries of Proposition \ref{prop211}.
\begin{corollary}
If $W_\alpha$ is a contractive weighted shift whose weights squared are completely alternating, then both the shift with weights $\delta_n = \sqrt{C(\alpha_n^2))}$ and the shift with weights $C(\alpha_n)$ have completely alternating weights, are contractive,  and are thus MID and subnormal. \    If $W_\alpha$ is a contractive weighted shift whose weights squared are log completely alternating (equivalently, its weights are log completely alternating) then the weighted shift $W_{GC(\alpha_n)}$ has weights log completely alternating and is thus MID and subnormal.
\end{corollary}

\begin{proof}
All is straightforward, keeping in mind that if a sequence is completely alternating so are its $p$-th power sequences for $p < 1$.
\end{proof}

Let $\Gamma(z)$ denote the classical Gamma function with $\Gamma'$ its derivative, and let $\gamma$ denote the Euler constant $\gamma \approx .577\ldots$.

\begin{corollary}
The weighted shift with weight sequence given by
$$\alpha_n = \sqrt{\frac{2 - \gamma + n - \frac{\Gamma'(3+n)}{\Gamma(3+n)}}{n+1}}, \hspace{.2in}n =0, 1, \ldots,$$
and as well the weighted shift with weight sequence given by the $\alpha_n^2$, are each moment infinitely divisible and hence subnormal.
\end{corollary}

\begin{proof}
The $\alpha_n^2$ sequence is the Ces\`aro transform of the weights squared sequence for the Bergman shift, as is shown by a computation.
\end{proof}

We remark that there are similar results obtainable by the Ces\`aro transform of the other Agler shifts. \  Note that in the following two propositions, the exponent ``$\gamma_n-1$'' ensures that a sequence has zeroth term $1$, as is required for a shift moment sequence.

\begin{proposition}  \label{prop:epowergamma}
If $(\gamma_n)$ is a positive completely monotone sequence with $\gamma_0 = 1$ then so is $(e^{\gamma_n -1})$, and therefore the shift with these moments is subnormal.
\end{proposition}

\begin{proof}
Consider some one of the expressions to be tested for $m$-monotonicity (that is, some one of the $(\nabla^m \gamma)(n))$, expand using the customary power series for $e^x$, and use that the set of completely monotone sequences is closed under the taking of positive integer powers.
\end{proof}

Similarly, we obtain

\begin{proposition}
If $(\gamma_n)$ is a positive infinitely divisible sequence with $\gamma_0 = 1$ then so is $(e^{\gamma_n-1})$, and so the shift with these moments is moment infinitely divisible.
\end{proposition}

\medskip

We record the following easy observation, in part because it leads to an open question (cf. Subsection \ref{concluding}.2).

\begin{corollary}
Suppose $(\alpha_n)$ is a positive, bounded, completely alternating sequence, and let $M = \sup \alpha_n$.  Then the shift with weights $\frac{e^{\alpha_n}}{e^M}$ is moment infinitely divisible.
\end{corollary}

\begin{proof}
The weights of the new (contractive) shift are obviously log completely alternating.
\end{proof}

\section{Complete hyperexpansivity and moment infinite divisibility} \label{complete}

We recall that a completely hyperexpansive weighted shift with weight sequence $(\alpha_n)_{n=0}^\infty$ gives rise to a subnormal weighted shift by forming a new weight sequence $\delta$ where $\delta_n = \frac{1}{\alpha_n}$ for all $n$;  further, one cannot necessarily begin with a subnormal shift, and, by taking reciprocals of the weights, generate a completely hyperexpansive shift. \ (See the discussion after \cite[Proposition 6]{At}.)

The following sheds some, but not complete, light on this result.

\begin{corollary}
Let $W_\alpha$ be a completely hyperexpansive weighted shift with positive weight sequence $(\alpha_n)_{n=0}^\infty$. \ Then the weighted shift with weight sequence $(\frac{1}{\alpha_n})_{n=0}^\infty$ is not only subnormal but moment infinitely divisible.
\end{corollary}

\begin{proof}
Let $(\gamma_n)$ be the moment sequence for $W_\alpha$;  since $W_\alpha$ is completely hyperexpansive, the sequence $(\gamma_n)$ is completely alternating. \ It is therefore log completely alternating. \ By a computation similar to what was done earlier, one shows that the sequence $(\alpha_n^2)$ (or, equivalently, $(\alpha_n)$) is what should naturally be called log completely monotone. It also results that each of the $\alpha_n$ must satisfy $\alpha_n \geq 1$, which is not surprising since $W_\alpha$ is an ``expansion.'' \ (A proof of this is elementary using the standard recurrence relationship for difference sequences.) \ It then follows easily that the sequence $(\frac{1}{\alpha_n^2})$ (equivalently, $(\frac{1}{\alpha_n})$) is log completely alternating and that each term is no larger than $1$, and that therefore the reciprocal shift is infinitely divisible (and therefore, of course, subnormal).
\end{proof}

Note that infinite divisibility of a subnormal shift is a necessary condition for its reciprocal shift to be completely hyperexpansive, but it is not sufficient, as shown by Athavale's example in the discussion cited above (this is just the third Agler shift, which we know is infinitely divisible).  Thus there is some proper subset of the completely alternating sequences which has some stronger property (see the condition in \cite[Proposition 6]{At}, which is another ``for all $t$'' type condition). Weights in $\Exp \mathcal{CA}$ is sufficient for infinite divisibility of the shift, but we do not know what more the weights in the more restrictive class $\mathcal{CA}$ tells you about the shift itself, nor do we know what even more the weights in this (yet) more restrictive Athavale class tells you about the shift itself.

Observe in passing that there is no hope of something like infinite divisibility for completely hyperexpansive shifts. \ Recall that the weights of a completely hyperexpansive shift are larger than $1$ (except in the trivial case of the unilateral shift), and the condition for $2$-expansivity includes
$$1 - 2 \alpha_0^2 + \alpha_0^2 \alpha_1^2 \leq 0.$$
Clearly if we change this to
$$1 - 2 {\alpha_0^{2p}} + {\alpha_0^{2p}} {\alpha_1^{2p}} \leq 0$$
there is no hope of it being satisfied for all $p > 0$. We know the set of $p$ for which it is satisfied includes $(0,1]$;  it still makes sense to ask what that set of $p$ is, and, in particular, if it is connected, just as the similar question for a subnormal operator not infinitely divisible makes good sense.

\section{Concluding remarks and open questions} \label{concluding}

We close with some remarks and questions. \

\begin{itemize}

\medskip
\item[1.] \label{item1} Notice that the sequences $\left( (\frac{n+1}{n+2})^2\right)$, $\left((\frac{n+1}{n+2})^3\right)$, $\left((\frac{n+1}{n+2})^4\right)$, $\cdots$, are not completely alternating. \ Now, if for a moment we think of completely alternating sequences as interpolated by  Bernstein functions (see \cite{BCR}), the Fa\`a di Bruno formulae \cite[Section 3.4, Theorem A, p. 137]{Com} shows that powers of the Bernstein functions are not necessarily Bernstein functions, and this suggests that powers of completely alternating sequences are not necessarily completely alternating. \ In fact, the sequence $\left((\frac{n+1}{n+2})^2\right)$ is $8$-alternating but not $9$-alternating, the sequence $\left((\frac{n+1}{n+2})^3\right)$ is $3$-alternating but not $4$-alternating, the sequence $\left((\frac{n+1}{n+2})^4\right)$ is $2$-alternating but not $3$-alternating, and the sequence $\left((\frac{n+1}{n+2})^5\right)$ is not even $2$-alternating.

\medskip
\item[2.] Consider now the weight sequence $\left(\frac{n+1}{n+2}\right)$, i.e., the Bergman weights squared. \ Since this is completely alternating it is obvious that the weight sequence $(e^{\frac{n+1}{n+2}})$ is log completely alternating and therefore the shift with weights $(e^{\frac{n+1}{n+2} - 1})$ is MID. \ However, there is evidence from \cite{Wol} that this second weight sequence is actually completely alternating (and not merely log completely alternating). \ Is this true? \ One may expand the exponential in a power series as in the proof of Proposition \ref{prop:epowergamma}, but since the powers of weights are greater than one in that series (not less than one) we may not cite the usual fact about $p$-th roots for completely alternating sequences to handle terms of degree $2$ or higher. \  In fact, such sequences are \textit{not} completely alternating (as shown in item 1 above). \ Is the result true nonetheless? \ We remark that the weight sequence $(e^{e^{\frac{n+1}{n+2}}})$ is not completely alternating (and it is log completely alternating precisely if the one considered just above is completely alternating).

\medskip
\item[3.] We know that if the weights of $W_\alpha$ are completely alternating they are log completely alternating, and so the shift is MID, but we do not know what ``extra'' results from being not only in $\mbox{\rm Exp}\,  \mathcal{CA}$ but in the more restrictive class $\mathcal{CA}_+$.

\medskip
\item[4.] There is another approach to showing that a weighted shift is subnormal, using Berger measures. \  Recall that every subnormal weighted shift $W_\alpha$ has an associated Berger measure $\mu$, a probability measure supported on $[0, \|W_\alpha\|^2]$ and satisfying
$$\gamma_n = \int_0^{\|W_\alpha\|^2} t^n d \mu(t), \hspace{.2in} n
= 0, 1, \ldots .$$
Further, if there exists such a measure matching the moments of $W_\alpha$, then the shift is subnormal. \  In \cite{CoLo} the proof of subnormality of the shift referenced in Corollary \ref{cor:Toeplitz} is by providing the Berger measure (it is a countably atomic measure with point masses at, as might be expected, $p^k$ for $k = 0,1, 2, \ldots$), and in \cite{CD} the Berger measures are provided for the shifts $S(a,b,c,d)$. \  The Berger measures for the Agler shifts, and a handful of other examples, are also known. \  It is a theme of \cite{CE} that finding such measures may be difficult:  in that paper it is obtained (using the Laplace transform) that the Berger measure of the weighted shift whose moment sequence is $\left(\sqrt[q]{\frac{1}{n+1}}\right)_{n=1}^\infty$ is $\frac{1}{\Gamma(q)} (-\ln u)^{q-1} \, du.$  This latter result is the only one of which we are aware in which the full family of Berger measures associated with a MID weighted shift (here, the Bergman shift) is known, and our characterization of the MID weighted shifts points to this new, and probably difficult, line of investigation.
\end{itemize}

\bigskip
{\bf Acknowledgment}. \ The authors wish to express their gratitude to an anonymous referee for detecting an omission in the original statement of Theorem 3.1.

\end{document}